\documentclass[12pt]{amsart}

\usepackage[matrix,arrow,curve,frame]{xy}
\usepackage{amsmath,amsthm,amssymb,enumerate}
\usepackage{latexsym}
\usepackage{amscd}
\usepackage[colorlinks=false]{hyperref}
\usepackage[mathscr]{eucal}

\newtheorem{thm}[subsection]{Theorem}

\newtheorem{cor}[subsection]{Corollary}
\newtheorem{lemma}[subsection]{Lemma}

\newtheorem{remark}[subsection]{Remark}

\theoremstyle{definition}

\numberwithin{equation}{section}

\def\phi{{\varphi}}

\def\ra{\rightarrow}

\def\cA{{\mathcal A}}

\def\cI{{\mathcal I}}
\def\cJ{{\mathcal J}}

\def\cV{{\mathcal V}}
\def\cW{{\mathcal W}}

\def\gg{{\mathfrak g}}

\def\gl{{\mathfrak l}}

\def\gs{{\mathfrak s}}

\newfont{\german}{eufm10}

\begin{document}
\pagestyle{plain}

\title
{The $\mathbb{Z}_2$-orbifold of the $\cW_3$-algebra}

\author{Masoumah Al-Ali}
\address{Department of Mathematics, University of Denver}
\email{masoumah.alali@du.edu}

\author{Andrew R. Linshaw} 
\address{Department of Mathematics, University of Denver}
\email{andrew.linshaw@du.edu}
\thanks{A. R. L. is supported by Simons Foundation Grant \#318755}

%\date{\today}
\begin{abstract} The Zamolodchikov $\cW_3$-algebra $\cW^c_3$ with central charge $c$ has full automorphism group $\mathbb{Z}_2$. It was conjectured in the physics literature over 20 years ago that the orbifold $(\cW^c_3)^{\mathbb{Z}_2}$ is of type $\cW(2,6,8,10,12)$ for generic values of $c$. We prove this conjecture for all $c \neq  \frac{559 \pm 7 \sqrt{76657}}{95}$, and we show that for these two values, the orbifold is of type $\cW(2,6,8,10,12,14)$. This paper is part of a larger program of studying orbifolds and cosets of vertex algebras that depend continuously on a parameter. Minimal strong generating sets for orbifolds and cosets are often easy to find for generic values of the parameter, but determining which values are generic is a difficult problem. In the example of $(\cW^c_3)^{\mathbb{Z}_2}$, we solve this problem using tools from algebraic geometry.
\end{abstract}

%\keywords{}
\maketitle
%\tableofcontents
\section{Introduction}
Given a vertex algebra $\cV$ and a group $G$ of automorphisms of $\cV$, the invariant subalgebra $\cV^G$ is called an {\it orbifold} of $\cV$. Many interesting vertex algebras can be constructed either as orbifolds or as extensions of orbifolds. A spectacular example is the Moonshine vertex algebra $V^{\natural}$, which is an extension of the $\mathbb{Z}_2$-orbifold of the lattice vertex algebra associated to the Leech lattice \cite{B,FLM}. There is a substantial literature on the structure and representation theory of orbifolds under finite group actions; see for example \cite{DVVV,DHVW,DM,DLMI,DLMII,DRX}. It is widely believed that nice properties of $\cV$ such as $C_2$-cofiniteness and rationality will be inherited by $\cV^G$ when $G$ is finite. In the case where $G$ is cyclic, the $C_2$-cofiniteness of $\cV^G$ was proven by Miyamoto in \cite{M}, and the rationality was recently established by Carnahan and Miyamoto in \cite{CM}.

Many vertex algebras depend continuously on a complex parameter $k$. Examples include the universal affine vertex algebra $V^k(\gg)$ associated to a simple, finite-dimensional Lie algebra $\gg$, and the $\cW$-algebra $\cW^k(\gg,f)$ associated to $\gg$ together with a nilpotent element $f\in \gg$. Typically, if $\cV^k$ is such a vertex algebra depending on $k$, it is simple for generic values of $k$ but has a nontrivial maximal proper ideal $\cI_k$ for special values. Often, one is interested in the structure and representation theory of the simple quotient $\cV_k = \cV^k / \cI_k$ at these points. For example, the $C_2$-cofiniteness and rationality of simple affine vertex algebras at positive integer level was proven by Frenkel and Zhu in \cite{FZ}, and the $C_2$-cofiniteness and rationality of several families of $\cW$-algebras is due to Arakawa \cite{A}.

Suppose that $\cV^k$ is such a vertex algebra and $G \subset \text{Aut}(\cV^k)$ is a reductive group of automorphisms such that $\cV^k$ decomposes as a sum of finite-dimensional $G$-modules. Then $G$ preserves $\cI_k$, and hence acts on $\cV_k$. For the purpose of studying the discrete family of orbifolds $(\cV_k)^G$ when $\cI_k$ is nontrivial in a uniform manner, it is useful to first consider the orbifold $(\cV^k)^G$ for generic values of $k$. If a strong generating set for $(\cV^k)^G$ can be found, it will descend to a strong generating set for $(\cV_k)^G$ since the projection $\pi_k: \cV^k \ra \cV_k$ restricts to a projection $(\cV^k)^G  \ra (\cV_k)^G$. Finding a strong generating set for a vertex algebra is very useful since strong generators give rise to generators for both Zhu's associative algebra and Zhu's commutative algebra \cite{Zh}.

In the case where $\cV^k = V^k(\gg)$, it was shown by the second author in \cite{L} that for any reductive group $G$, $V^k(\gg)^G$ is strongly finitely generated for generic values of $k$. This method can be adapted to study the generic behavior of other families of orbifolds, including orbifolds of affine vertex superalgebras \cite{CL} and minimal $\cW$-algebras \cite{ACKL}. Unfortunately, the approach of \cite{L} gives little insight into which values of $k$ are generic. In order to use the strong generating set for $(\cV^k)^G$ to obtain a strong generating set for $(\cV_k)^G$, it is necessary to know whether or not $k$ is a generic value.

In this paper, we give a complete solution to this problem for the $\mathbb{Z}_2$-orbifold of the Zamolodchikov $\cW_3$-algebra $\cW^c_3$ with central charge $c$. It was conjectured over 20 years ago in the physics literature \cite{BS,B-H} that $(\cW^c_3)^{\mathbb{Z}_2}$ should be of type $\cW(2,6,8,10,12)$ for generic values of $c$. In other words, a minimal strong generating set consists of one field in each weight $2,6,8,10,12$. Our main result is the following.

\begin{thm} \label{intro:mainthm}
\begin{enumerate}
\item For all $c \neq \frac{559 \pm 7 \sqrt{76657}}{95}$, $(\cW^c_3)^{\mathbb{Z}_2}$ is of type $\cW(2,6,8,10,12)$.
\item For $c =  \frac{559 \pm 7 \sqrt{76657}}{95}$, $(\cW^c_{3})^{\mathbb{Z}_2}$ is of type $\cW(2,6,8,10,12,14)$.
\end{enumerate}
\end{thm}

An immediate consequence is that for all $c \neq \frac{559 \pm 7 \sqrt{76657}}{95}$, the simple orbifold $(\cW_{3,c})^{\mathbb{Z}_2}$ has the same strong generating set, where $\cW_{3,c}$ denotes the simple quotient of $\cW^c_3$. If the maximal ideal $\cI_c \subset \cW^c_3$ has components of weight $w \leq 12$, there may be additional decoupling relations and this strong generating set for $(\cW_{3,c})^{\mathbb{Z}_2}$ need not be minimal. If explicit generators for $\cI_c$ are known, it is straightforward to reduce the above strong generating set to a minimal set if possible, but we do not carry this out in the present paper.

The proof of Theorem \ref{intro:mainthm} proceeds as follows. We first construct a natural infinite strong generating set $$\{L, U_{2n,0}|\ n\geq 0\}$$ for $(\cW^c_3)^{\mathbb{Z}_2}$, where $U_{2n,0} = \ :(\partial^{2n} W)W:$, which has weight $2n+6$. This generating set comes from classical invariant theory, and there are infinitely many nontrivial normally ordered relations among these generators. The relation of minimal weight $14$ is unique up to scalar multiples, and has the form
$$\frac{181248 + 5590 c - 475 c^2}{60480 (22 + 5 c)} U_{8,0} = P(L, U_{0,0}, U_{2,0}, U_{4,0}, U_{6,0}),$$ where $P$ is a normally ordered polynomial in $L, U_{0,0}, U_{2,0}, U_{4,0}, U_{6,0}$ and their derivatives. The pole at $c = -\frac{22}{5}$ is inessential and can be removed; it is a consequence of the choice of normalization of $W$. Therefore $U_{8,0}$ can be eliminated if and only if $c \neq \frac{559 \pm 7 \sqrt{76657}}{95}$. In a similar way, we construct decoupling relations for all $c$ expressing $U_{10,0}$, $U_{12,0}$ and $U_{14,0}$ as normally ordered polynomials in $L, U_{0,0}, U_{2,0}, U_{4,0}, U_{6,0}, U_{8,0}$, and their derivatives. In order to construct decoupling relations for $U_{n,0}$ for all even integers $n\geq 16$, we apply the operators $U_{0,0} \circ_1$ and $U_{2,0}\circ_1$ to the above relations. This yields two families of relations
\begin{equation} \label{intro:decoup1} F(n,c) U_{n+4,0} = A_n(L, U_{0,0}, U_{2,0},\dots, U_{n+2,0}),\end{equation}
\begin{equation} \label{intro:decoup2} G(n,c) U_{n+6,0} = B_n(L, U_{0,0}, U_{2,0},\dots, U_{n+4,0}),\end{equation} where $A_n$ and $B_n$ are normally ordered polynomials as above. The key observation is that $F(n,c)$ and $G(n,c)$ are rational functions of $c$ and $n$ which have no poles for $c \neq -\frac{22}{5}$ and $n \geq 10$. So if $(c,n)$ does not lie on the affine variety $V \subset \mathbb{C}^2$ determined by $F(n-4,c) = 0 $ and $G(n-6,c) = 0$, we can use either \eqref{intro:decoup1} or \eqref{intro:decoup2} to eliminate $U_{n,0}$ for all $n\geq 16$. The main technical result in this paper is finding the explicit form of $F(n,c)$ and $G(n,c)$; it is then straightforward to prove that $V$ has no such points $(c,n)$ where $n \geq 16$ is an even positive integer.

Although this result may at first seem to be an isolated case study, it in fact provides a general algorithmic approach to determining the nongeneric set for orbifolds of the form $(\cV^k)^G$. Typically, there is a natural infinite strong generating set for $(\cV^k)^G$ coming from classical invariant theory. There are also infinitely many nontrivial normally ordered relations among these generators. These relations allow certain generators to be eliminated, and for generic values of $k$, all but finitely many can be eliminated. If we eliminate as many generators as possible, the remaining ones will form a minimal strong generating set $S$, and a value of $k$ will be called {\it generic} if $(\cV^k)^G$ is strongly generated by $S$. We expect that families of relations can be constructed such that the coefficients of the generators to be eliminated are rational functions in finitely many variables $$F_i(k, n_1,\dots, n_r),\qquad i =1,\dots, s.$$ Here $n_1,\dots, n_r$ must be positive integers, and are related to the weights of the generators to be eliminated. Corresponding to such a system of relations is the variety $V \subset \mathbb{C}^{r+1}$ determined by $F_i(k, n_1,\dots, n_r) = 0$ for $i=1,\dots, s$. A value of $k$ will be generic if there is no point $(k, n_1,\dots, n_r)\in V$ such that the remaining coordinates $n_1,\dots, n_r$ are positive integers. Points with such strong integrality constraints are expected to be rare, and in principle can be found. It would be nice to prove in some generality that the nongeneric set for orbifolds of the form $(\cV^k)^G$ is always finite, but this is out of reach at the moment.

\section{Vertex algebras}
We will assume that the reader is familiar with the basics of vertex algebra theory, which has been discussed from several different points of view in the literature (see for example \cite{B,FLM,K,FBZ}). We will follow the formalism developed in \cite{LZ} and partly in \cite{LiI}, and we will use the notation of the previous papers of the second author \cite{ACL,CL,L}. Given an element $a$ in a vertex algebra $\cA$, the corresponding field is denoted by $$a(z) = \sum_{n\in \mathbb{Z}} a(n) z^{-n-1}.$$ Given $a,b \in \cA$, the {\it operator product expansion} (OPE) formula is given by
$$a(z)b(w)\sim\sum_{n\geq 0}(a \circ_n b)(w)\ (z-w)^{-n-1}.$$ Here $(a \circ_n b)(w) = \text{Res}_z [a(z), b(w)](z-w)^n$ where $[a(z), b(w)] = a(z) b(w)\ -\ (-1)^{|a||b|} b(w)a(z)$, and $\sim$ means equal modulo terms which are regular at $z=w$. The normally ordered (or Wick) product $:a(z)b(z):$ is defined to be $$a(z)_-b(z)\ +\ (-1)^{|a||b|} b(z)a(z)_+,$$ where $$a(z)_-=\sum_{n<0}a(n)z^{-n-1},\qquad a(z)_+=\sum_{n\geq
0}a(n)z^{-n-1}.$$  For fields $a_1(z),\dots, a_k(z)$, the $k$-fold
iterated Wick product is defined inductively to be
\begin{equation} \label{iteratedwick} :a_1(z)a_2(z)\cdots a_k(z):\ =\ :a_1(z) \big(  :a_2(z)\cdots a_k(z): \big).\end{equation}
A subset $S=\{a_i|\ i\in I\}$ of $\cA$ {\it strongly generates} $\cA$ if $\cA$ is spanned by $$\{ :\partial^{k_1} a_{i_1}(z)\cdots \partial^{k_m} a_{i_m}(z):| \ i_1,\dots,i_m \in I,\ k_1,\dots,k_m \geq 0\}.$$ We say that $S$ {\it freely generates} $\cA$ if there are no nontrivial normally ordered polynomial relations among the generators and their derivatives. We say that $\cA$ is of type $$\cW(d_1,\dots, d_r)$$ if it has a minimal strong generating set consisting of one field in each weight $d_1,\dots, d_r$.

\section{The $\cW_3$-algebra}
The $\cW_3$-algebra $\cW^c_3$ with central charge $c$ was introduced by Zamolodchikov \cite{Za}. It is a nonlinear extension of the Virasoro algebra of type $\cW(2,3)$, and is strongly generated by a Virasoro field $L$ and a weight $3$ primary field $W$ satisfying
\begin{equation} \label{w3first} L(z) L(w) \sim \frac{c}{2} (z-w)^{-4} + 2L(w)(z-w)^{-2} + \partial L(w)(z-w)^{-1},\end{equation}
\begin{equation} \label{w3second} L(z) W(w) \sim 3 W(w)(z-w)^{-2} + \partial W(w)(z-w)^{-1},\end{equation}
\begin{equation} \label{w3third} \begin{split} W(z)W(w) \sim \frac{c}{3}(z-w)^{-6} + 2L(w)(z-w)^{-4} + \partial L(w)(z-w)^{-3} \\ + \bigg( \frac{32}{22 + 5 c} :LL: + \frac{3 (-2 + c)}{2 (22 + 5 c)} \partial^2 L \bigg) (z-w)^{-2} \\
+ \bigg( \frac{32}{22 + 5 c} :(\partial L)L: + \frac{-2 + c}{3 (22 + 5 c)} \partial^3 L \bigg) (z-w)^{-1}. \end{split}\end{equation} 
In fact, $\cW^c_3$ is isomorphic to the principal $\cW$-algebra $\cW^k(\gs\gl_3, f_{\text{prin}})$ where $c = 2 - \frac{24(k+2)^2}{k+3}$. Even though \eqref{w3third} has a pole at $c = -\frac{22}{5}$, we can still define $\cW^{-22/5}_3$ by rescaling $W$ by a factor of $\sqrt{22 + 5 c}$, and then taking the limit as $c \ra -\frac{22}{5}$. The rescaled generator, also denoted by $W$, now satisfies
\begin{equation} \label{ope:c=-22/5} W(z)W(w) \sim \bigg(32 :LL: -\frac{48}{5} \partial^2 L\bigg) (z-w)^{-2} + \bigg(32 :(\partial L)L:  - \frac{32}{15} \partial^3 L \bigg)(z-w)^{-1}.\end{equation}
For all $c\in \mathbb{C}$, $\cW^c_3$ is freely generated by $L,W$; in particular, it has a PBW basis
\begin{equation} \label{standardmonomial} :(\partial^{a_1} L) \cdots (\partial^{a_r} L) (\partial^{b_1} W) \cdots (\partial^{b_{s}} W):, \quad r,s\geq 0,\quad a_1\geq \cdots \geq a_r \geq 0,\quad b_1 \geq \cdots \geq b_s \geq 0.\end{equation} For simplicity, we shall use the notation $\cW$ for $\cW^c_3$ throughout this paper.

\subsection{Filtrations}
In \cite{ACL}, the notion of a {\it weak increasing filtration} on a vertex algebra $\cA$ was introduced. It is a $\mathbb{Z}_{\geq 0}$-filtration
\begin{equation}\label{weak} \cA_{(0)}\subset\cA_{(1)}\subset\cA_{(2)}\subset \cdots,\qquad \cA_{(-1)} = \{0\}, \qquad \cA = \bigcup_{d\geq 0}
\cA_{(d)}, \end{equation} such that for $a\in \cA_{(r)}$, $b\in\cA_{(s)}$, we have
\begin{equation} a\circ_n b \in \cA_{r+s},\qquad n\in \mathbb{Z}.\end{equation}
This condition guarantees that $\text{gr}(\cA) = \bigoplus_{d\geq 0}\cA_{(d)}/\cA_{(d-1)}$ is a vertex algebra. Let $$\phi_d: \cA_{(d)} \ra \cA_{(d)} / \cA_{(d-1)} \subset \text{gr}(\cA)$$ be the projection. As in the case of good increasing filtrations \cite{LiII}, we have the following reconstruction property, and the proof is the same as the proof of Lemma 3.6 of \cite{LL}.

\begin{lemma}\label{reconlem}Let $\cA$ be a vertex algebra with a weak increasing filtration, and let $\{a_i|\ i\in I\}$ be a set of strong generators for $\text{gr}(\cA)$, where $a_i$ is homogeneous of degree $d_i$. If $\tilde{a}_i\in\cA_{(d_i)}$ are elements satisfying $\phi_{d_i}(\tilde{a}_i) = a_i$, then $\cA$ is strongly generated as a vertex algebra by $\{\tilde{a}_i|\ i\in I\}$.\end{lemma}

The filtration $\cW_{(0)} \subset \cW_{(1)} \subset \cdots$ on $\cW$ is defined as follows: $\cW_{(-1)} = \{0\}$, and $\cW_{(r)}$ is spanned by iterated Wick products of the generators $L,W$ and their derivatives, such that at most $r$ copies of $W$ and its derivatives appear. It is clear from \eqref{w3first}-\eqref{w3third} that this is a weak increasing filtration, and $\cW_{(0)}$ is the Virasoro algebra with generator $L$. Note that the associated graded algebra $$\cV = \text{gr}(\cW) = \bigoplus_{d \geq 0} \cW_{(d)} / \cW_{(d-1)}$$ is freely generated by $L,W$. The OPE relations \eqref{w3first}-\eqref{w3second} still hold in $\cV$, but \eqref{w3third} is replaced with $W(z) W(w) \sim 0$. Finally, $\cV$ has a good increasing filtration $$\cV_{(0)} \subset \cV_{(1)} \subset \cdots,$$ where  $\cV_{(-1)} = \{0\}$, and $\cV_{(r)}$ is spanned by iterated Wick products of the generators $L,W$ and their derivatives, of length at most $r$. Then $\text{gr}(\cV)$ is an abelian vertex algebra freely generated by $L,W$. In particular, $\text{gr}(\cV)$ is isomorphic to the polynomial algebra $$\mathbb{C}[L,\partial L, \partial^2 L,\dots, W, \partial W,\partial^2 W,\dots].$$

\section{The $\mathbb{Z}_2$-orbifold of $\cW$} \label{sect:structureorbifold}
The full automorphism group of $\cW$ is $\mathbb{Z}_2$, where the nontrivial involution $\theta$ acts on the generators as follows:
\begin{equation}\label{action} \theta(L) = L,\qquad \theta(W) = -W.\end{equation}

It is immediate that  $\cW^{\mathbb{Z}_2}$ is spanned by all normally ordered monomials of the form \eqref{standardmonomial}, where $s$ is even. We say that $\omega \in \cW^{\mathbb{Z}_2}$ is in {\it normal form} if it has been expressed as a linear combination of such monomials. Since $\cW$ is freely generated by $L,W$, these monomials form a basis for $\cW^{\mathbb{Z}_2}$, and the normal form is unique.

The filtration on $\cW$ restricts to a filtration on $\cW^{\mathbb{Z}_2}$, $$\cW^{\mathbb{Z}_2}_{(0)} \subset\cW^{\mathbb{Z}_2}_{(1)} \subset \cdots ,\qquad \cW^{\mathbb{Z}_2}_{(r)} = \cW^{\mathbb{Z}_2} \cap \cW_{(r)} .$$ 
The $\mathbb{Z}_2$-action descends to $\cV = \text{gr}(\cW)$, and $$\text{gr}(\cW^{\mathbb{Z}_2}) \cong \cV^{\mathbb{Z}_2}.$$ Similarly, $\mathbb{Z}_2$ acts on $\text{gr}(\cV)  \cong \mathbb{C}[L,\partial L, \partial^2 L,\dots, W, \partial W,\partial^2 W,\dots]$, and $$\text{gr}(\cV^{\mathbb{Z}_2}) \cong \text{gr}(\cV)^{\mathbb{Z}_2} \cong \mathbb{C}[L,\partial L, \partial^2 L,\dots, W, \partial W,\partial^2 W,\dots]^{\mathbb{Z}_2}.$$ Since the action is given by
$\theta(\partial^k L) = \partial^k L$ and $\theta(\partial^k W) = - \partial^k W$, $\text{gr}(\cV)^{\mathbb{Z}_2}$ is generated by
$\{L, u_{i,j}|\ i, j\geq 0\}$, where $u_{i,j} = (\partial^i W) (\partial^j W)$. Note that $u_{i,j}= u_{j,i}$, so we only need $\{L, u_{i,j}|\ i \geq j \geq 0\}$. The ideal of relations among these generators is clearly generated by \begin{equation} \label{ideal} u_{i,j} u_{k,l} - u_{i,l} u_{k,j}=0,\qquad 0 \leq i<k,\qquad 0 \leq j<l.\end{equation}
As a differential algebra with derivation $\partial$, there is some redundancy in this generating set for $\text{gr}(\cV)^{\mathbb{Z}_2} $ since $$ \partial u_{i,j} = u_{i+1,j} +  u_{i,j+1}.$$ Letting $A_n$ be the span of $\{u_{i,j}|\ i+j = n\}$, note that $A_n = \partial (A_{n-1})$ if $n$ is odd, and $A_n = \partial(A_{n-1}) \oplus \langle u_{n,0} \rangle$ if $n$ is even. Therefore 
$\{u_{i,j}|\ i,j \geq 0\}$ and $\{\partial^m u_{2n,0}|\ m,n\geq 0\}$ span the same vector space, and $\{L, u_{2n,0}|\ n\geq 0\}$ is a minimal generating set for $\text{gr}(\cV)^{\mathbb{Z}_2}$ as a differential algebra. Define \begin{equation} \label{defuij} U_{i,j}=\ :(\partial^i W)( \partial^j W):\ \in \cW^{\mathbb{Z}_2}_{(2)},\end{equation} which has filtration degree $2$ and weight $i+j+6$. 
\begin{lemma} $\cW^{\mathbb{Z}_2}$ is strongly generated by
\begin{equation} \label{gen} \{L, U_{2n,0}|\ n\geq 0\}.\end{equation}
\end{lemma}
\begin{proof} Since $\{L, u_{2n,0}|\ n\geq 0\}$ generates $\text{gr}(\cV)^{\mathbb{Z}_2} \cong \text{gr}(\cV^{\mathbb{Z}_2})$ as a differential algebra, Lemma 3.6 of \cite{LL} shows that the corresponding set strongly generates $\cV^{\mathbb{Z}_2}$ as a vertex algebra. Applying Lemma \ref{reconlem} then yields the result. In particular, note that in $\cW^{\mathbb{Z}_2}$, both $U_{i,j} - U_{j,i}$ and $\partial U_{i,j} - U_{i+1,j} - U_{i,j+1}$ can be expressed as normally ordered polynomials in $L$ and its derivatives, so $\{U_{i,j}|\ i,j \geq 0\}$ and $\{\partial^m U_{2n,0}|\ m,n\geq 0\}$ span the same vector space, modulo the Virasoro algebra generated by $L$.
\end{proof} 
\begin{remark} $U_{2n,0}$ is not primary with respect to $L$. For the first few values of $n$, it is easy to correct $U_{2n,0}$ by adding a normally ordered polynomial in $L, U_{0,0}, U_{2,0} \dots, U_{2n-2,0}$ and their derivatives to make it primary, but this is not necessary for our purposes.\end{remark}
\begin{remark} In terms of the generating set \eqref{gen}, $\cW^{\mathbb{Z}_2}_{(2r)}$ is spanned by elements with at most $r$ of the fields $U_{2n,0}$, and 
$\cW^{\mathbb{Z}_2}_{(2r)} = \cW^{\mathbb{Z}_2}_{(2r+1)}$. \end{remark}

\section{Decoupling relations} \label{sect:decouplingrel} Observe next that $\cW^{\mathbb{Z}_2}$ is not {\it freely} generated by \eqref{gen}. To see this, observe that 
\begin{equation} \label{rel:wt14cl} u_{0,0} u_{1,1} - u_{1,0} u_{1,0}=0 \end{equation} is the unique relation of the form \eqref{ideal} in $\text{gr}(\cV)^{\mathbb{Z}_2}$, of minimal weight $14$. The corresponding element $:U_{0,0} U_{1,1}: - :U_{1,0} U_{1,0}:$ of $\cW^{\mathbb{Z}_2}$ does not vanish due to \eqref{w3third}. However, it lies in the degree $2$ filtered piece $\cW^{\mathbb{Z}_2}_{(2)}$ and has the form 
\begin{equation} \label{rel:wt14} :U_{0,0} U_{1,1}: - :U_{1,0} U_{1,0}: \ = \frac{181248 + 5590 c - 475 c^2}{60480 (22 + 5 c)} U_{8,0} + P(L, U_{0,0}, U_{2,0}, U_{4,0}, U_{6,0}),\end{equation} where $P$ is a normally ordered polynomial in $L, U_{0,0}, U_{2,0}, U_{4,0}, U_{6,0}$, and their derivatives. For the reader's convenience, this relation is written down explicitly in the Appendix. Note that $$U_{1,0} = \frac{1}{2} \partial U_{0,0}  - \frac{8}{3 (22 + 5 c)} : (\partial^3 L)L: - \frac{8}{22 + 5 c} :(\partial^2 L)\partial L: - \frac{-2 + c}{48 (22 + 5 c)} \partial^5 L,$$
$$U_{1,1} =  - U_{2,0} + \frac{1}{2} \partial^2 U_{0,0}  - \frac{8}{3 (22 + 5 c)} :(\partial^4 L)L: - \frac{32}{3 (22 + 5 c)} :(\partial^3 L)\partial L: $$ $$ -  \frac{8}{22 + 5 c} :(\partial^2 L)\partial^2 L: + \frac{2 - c}{48 (22 + 5 c)} \partial^6 L:.$$ Therefore the left side of \eqref{rel:wt14} is a normally ordered polynomial in $L, U_{0,0}, L_{2,0}$, so \eqref{rel:wt14} can be written in the form
\begin{equation} \label{rel:wt14a} \frac{181248 + 5590 c - 475 c^2}{60480 (22 + 5 c)} U_{8,0} = P_8(L, U_{0,0}, U_{2,0}, U_{4,0}, U_{6,0}).\end{equation}
We call this a {\it decoupling relation} since it allows $U_{8,0}$ to be expressed as a normally ordered polynomial in $L, U_{0,0}, U_{2,0}, U_{4,0}, U_{6,0}$ and their derivatives whenever $c \neq -\frac{22}{5}, \frac{559 \pm 7 \sqrt{76657}}{95}$. In fact, the pole at $c = -\frac{22}{5}$ is inessential and is a consequence of the choice of normalization of $W$. For convenience, we shall assume that $c \neq -\frac{22}{5}$ for the remainder of Sections \ref{sect:decouplingrel}-\ref{sect:commonzero}, and we deal with the case $c = -\frac{22}{5}$ separately in Section \ref{sect:22/5}.

Since there are no relations in $\text{gr}(\cV)^{\mathbb{Z}_2}$ of weight less than $14$, there are no decoupling relations for $U_{0,0}, U_{2,0}, U_{4,0}, U_{6,0}$. We shall see that the coefficient of $U_{8,0}$ in \eqref{rel:wt14a} is {\it canonical} in the sense that it does not depend on any choices of normal ordering in $P_8$. The uniqueness of \eqref{rel:wt14cl} up to scalar multiples implies the uniqueness of \eqref{rel:wt14a}, so for $c = \frac{559 \pm 7 \sqrt{76657}}{95}$, there is no decoupling relation for $U_{8,0}$.

\subsection{Weight $16$ relations} By correcting the relation $u_{0,0} u_{2,2} - u_{2,0} u_{2,0}=0$ in $\text{gr}(\cV)^{\mathbb{Z}_2}$ as above, we get the following relation in $\cW^{\mathbb{Z}_2}$ in weight $16$:
\begin{equation} \label{rel:wt1622}  :U_{0,0} U_{2,2}: - :U_{2,0} U_{2,0}:\  = -\frac{434176 - 20326 c + 35 c^2}{151200 (22 + 5 c)} U_{10,0} + Q(L,U_{0,0}, U_{2,0}, U_{4,0}, U_{6,0}, U_{8,0}),\end{equation} where $Q$ is a normally ordered polynomial in $L,U_{0,0}, U_{2,0}, U_{4,0}, U_{6,0}, U_{8,0}$ and their derivatives. As above, $U_{2,2}$ can be written as a normally ordered polynomial in $L, U_{0,0}, U_{2,0}, U_{4,0}$ and their derivatives, so \eqref{rel:wt1622} can be written in the form
\begin{equation} \label{rel:wt1622a}-\frac{434176 - 20326 c + 35 c^2}{151200 (22 + 5 c)} U_{10,0} = Q_{10} (L,U_{0,0}, U_{2,0}, U_{4,0}, U_{6,0}, U_{8,0}).\end{equation} This shows that $U_{10,0}$ can be eliminated whenever $c \neq \frac{10163 \pm \sqrt{88090409}}{35}$. 

Similarly, by correcting the relation $u_{0,0} u_{3,1} - u_{3,0} u_{1,0} = 0$, we get another relation
\begin{equation} \label{rel:wt1631} :U_{0,0} U_{3,1}: - :U_{3,0} U_{1,0}:\  = -\frac{13 (-1920 - 42 c + 5 c^2)}{9450 (22 + 5 c)} U_{10,0} + Q'(L,U_{0,0}, U_{2,0}, U_{4,0}, U_{6,0}, U_{8,0}),\end{equation} which can be rewritten as
\begin{equation} \label{rel:wt1631a} -\frac{13 (-1920 - 42 c + 5 c^2)}{9450 (22 + 5 c)} U_{10,0} = Q'_{10}(L,U_{0,0}, U_{2,0}, U_{4,0}, U_{6,0}, U_{8,0}).\end{equation} 
This works for $c \neq \frac{21 \pm \sqrt{10041}}{5}$, so for all $c$, we can use either \eqref{rel:wt1622a} or \eqref{rel:wt1631a} to express $U_{10,0}$ as a normally ordered polynomial in $L,U_{0,0}, U_{2,0}, U_{4,0}, U_{6,0}, U_{8,0}$, and their derivatives. Note also that if $c \neq \frac{559 \pm 7 \sqrt{76657}}{95}$, we can use \eqref{rel:wt14a} to eliminate $U_{8,0}$ from either \eqref{rel:wt1622a} or \eqref{rel:wt1631a}, so we can rewrite these in the form
\begin{equation} \label{rel:wt1622b}-\frac{434176 - 20326 c + 35 c^2}{151200 (22 + 5 c)} U_{10,0}  = P_{10} (L,U_{0,0}, U_{2,0}, U_{4,0}, U_{6,0}),\end{equation} 
\begin{equation} \label{rel:wt1631b} -\frac{13 (-1920 - 42 c + 5 c^2)}{9450 (22 + 5 c)} U_{10,0} = P'_{10}(L,U_{0,0}, U_{2,0}, U_{4,0}, U_{6,0}).\end{equation}

\subsection{Weight $18$ relations}
By correcting the relation $u_{0,0} u_{3,3} - u_{3,0} u_{3,0} = 0$, we get
$$:U_{0,0} U_{3,3}: - :U_{3,0} U_{3,0}:\  = \frac{4012032 + 28306 c - 9625 c^2}{1663200 (22 + 5 c)} U_{12,0} + R(L,U_{0,0}, U_{2,0}, U_{4,0}, U_{6,0},U_{8,0},U_{10,0}).$$ Using either \eqref{rel:wt1622a} or \eqref{rel:wt1631a} to eliminate $U_{10,0}$ yields
\begin{equation} \label{rel:wt1833a}  \frac{4012032 + 28306 c - 9625 c^2}{1663200 (22 + 5 c)} U_{12,0} = Q_{12}(L,U_{0,0}, U_{2,0}, U_{4,0}, U_{6,0},U_{8,0}).\end{equation}
so $U_{12,0}$ can be eliminated for $c \neq \frac{14153 \pm \sqrt{38816115409}}{9625}$.

Similarly, correcting the relation $u_{0,0} u_{4,2} - u_{4,0} u_{2,0} = 0$, and eliminating $U_{10,0}$ yields
\begin{equation} \label{rel:wt1842a} \frac{-2785280 + 145762 c - 385 c^2}{1108800 (22 + 5 c)} U_{12,0} = Q'_{12}(L,U_{0,0}, U_{2,0}, U_{4,0}, U_{6,0},U_{8,0}),\end{equation} so $U_{12,0}$ can be eliminated for $c \neq \frac{72881 \pm \sqrt{4239307361}}{385}$. Therefore using either \eqref{rel:wt1833a} or \eqref{rel:wt1842a}, we can eliminate $U_{12,0}$ for all $c$. As above, if $c \neq \frac{559 \pm 7 \sqrt{76657}}{95}$, we can use \eqref{rel:wt14a} to eliminate $U_{8,0}$ from these equations, obtaining 
\begin{equation} \label{rel:wt1833b}  \frac{4012032 + 28306 c - 9625 c^2}{1663200 (22 + 5 c)} U_{12,0} = P_{12}(L,U_{0,0}, U_{2,0}, U_{4,0}, U_{6,0}),\end{equation}
\begin{equation} \label{rel:wt1842b} \frac{-2785280 + 145762 c - 385 c^2}{1108800 (22 + 5 c)} U_{12,0} = P'_{12}(L,U_{0,0}, U_{2,0}, U_{4,0}, U_{6,0}).\end{equation}

\subsection{Weight $20$ relations}
By correcting the relation $u_{0,0} u_{4,4} - u_{4,0} u_{4,0} = 0$, we get
$$:U_{0,0} U_{4,4}: - :U_{4,0} U_{4,0}:\  = \frac{-20559360 + 1209594 c - 5005 c^2}{9459450 (22 + 5 c)} U_{12,0} $$ $$+ S(L,U_{0,0}, U_{2,0}, U_{4,0},U_{6,0},U_{8,0},U_{10,0},U_{12,0}).$$ 
Eliminating $U_{10,0}$ and $U_{12,0}$ yields
\begin{equation} \label{rel:wt2044a} \frac{-20559360 + 1209594 c - 5005 c^2}{9459450 (22 + 5 c)}  U_{14,0} = Q_{14}(L,U_{0,0}, U_{2,0}, U_{4,0}, U_{6,0},U_{8,0}).\end{equation}
so $U_{14,0}$ can be eliminated for $c \neq \frac{604797 \pm \sqrt{262879814409}}{5005}$.

Similarly, correcting the relation $u_{0,0} u_{6,2} - u_{6,0} u_{2,0} = 0$, and eliminating $U_{10,0}$ and $U_{12,0}$ yields
\begin{equation} \label{rel:wt2062a} \frac{-26284032 + 1487354 c - 5005 c^2}{12108096 (22 + 5 c)} U_{14,0} = Q'_{14}(L,U_{0,0}, U_{2,0}, U_{4,0}, U_{6,0},U_{8,0}),\end{equation} so $U_{14,0}$ can be eliminated for $c \neq \frac{743677 \pm \sqrt{421503900169}}{5005}$. Therefore using either \eqref{rel:wt2044a} or \eqref{rel:wt2062a}, we can eliminate $U_{14,0}$ for all $c$. As above, if $c \neq \frac{559 \pm 7 \sqrt{76657}}{95}$, we can use \eqref{rel:wt14a} to eliminate $U_{8,0}$ from these equations, obtaining 
\begin{equation} \label{rel:wt2044b} \frac{-20559360 + 1209594 c - 5005 c^2}{9459450 (22 + 5 c)}  U_{14,0} = P_{14}(L,U_{0,0}, U_{2,0}, U_{4,0}, U_{6,0}),\end{equation}
\begin{equation} \label{rel:wt2062b} \frac{-26284032 + 1487354 c - 5005 c^2}{12108096 (22 + 5 c)} U_{14,0} = P'_{14}(L,U_{0,0}, U_{2,0}, U_{4,0}, U_{6,0}).\end{equation}

\section{Higher decoupling relations} \label{sect:higherdecoup}
The above calculations suggest that for all $c \neq \frac{559 \pm 7 \sqrt{76657}}{95}$, there exist higher decoupling relations 
$$U_{n,0} = P_n(L,U_{0,0}, U_{2,0}, U_{4,0}, U_{6,0}),\qquad n = 16, 18, 20, \dots,$$ and for $c = \frac{559 \pm 7 \sqrt{76657}}{95}$, there exist relations 
$$U_{n,0} = Q_n(L,U_{0,0}, U_{2,0}, U_{4,0}, U_{6,0}, U_{8,0}),\qquad n = 16, 18, 20,  \dots.$$ We shall construct these relations in a uniform manner by applying the operators $U_{0,0} \circ_1$ and $U_{2,0}\circ_1$ successively to the relations we have already constructed for $n = 8,10,12,14$. First, we need a certain invariant of elements of $\cW^{\mathbb{Z}_2}_{(2)}$ of even weight. Given $\omega \in \cW^{\mathbb{Z}_2}_{(2)}$ of weight $n+6$, where $n$ is an even integer, write $\omega$ in normal form. For $i = 0,1,\dots, \frac{n}{2}$, let 
\begin{equation} C_{n,i}(\omega)\end{equation} denote the coefficient of $:(\partial^{n-i} W)( \partial^i W):$ appearing in the normal form, which is well-defined by uniqueness of \eqref{standardmonomial}. Next, let  
\begin{equation} \label{defcn} C_n(\omega) = \sum_{i=0}^{n/2} (-1)^i C_{n,i}(\omega).\end{equation}
Since $\{L,U_{n,0}|\ n=0,2,4,\dots\}$ strongly generates $\cW^{\mathbb{Z}_2}$ and since $U_{n,0}$ has weight $n+6$, we may write $$\omega = P_{\omega}(L,U_{0,0}, U_{2,0},\dots, U_{n,0}),$$ where $P_{\omega}$ is a normally ordered polynomial in $L,U_{0,0},U_{2,0},\dots, U_{n,0}$, and their derivatives. 
 Since there exist normally ordered relations among these generators, as well as different choices of normal ordering, such an expression for $\omega$ is not unique. In particular, the coefficients of $ \partial^i U_{n-i,0}$ for $ i = 2,4,\dots, n$ will depend on the choice of $P_{\omega}$.

\begin{lemma} \label{indep} For any $\omega \in \cW^{\mathbb{Z}_2}_{(2)}$ of weight $n+6$, the coefficient of $U_{n,0}$ in $P_{\omega}$ is independent of all choices of normal ordering, and coincides with $C_n(\omega)$. \end{lemma} 

\begin{proof} Let $\cJ \subset \cW^{\mathbb{Z}_2}$ denote the subspace spanned by elements of the form $:a \partial b:$ with $a,b \in  \cW^{\mathbb{Z}_2}$. It is well known that Zhu's commutative algebra $C(\cW^{\mathbb{Z}_2}) =  \cW^{\mathbb{Z}_2} / \cJ$ is a commutative, associative algebra with generators corresponding to the strong generators $\{L, U_{2n,0}|\ n\geq 0\}$. In particular, given $\omega \in  \cW^{\mathbb{Z}_2}_{(2)}$ of filtration degree $2$ and even weight $n+6$, suppose that $$\omega = P_{\omega}(L,U_{0,0}, U_{2,0},\dots, U_{n,0}) =  Q_{\omega}(L,U_{0,0}, U_{2,0},\dots, U_{n,0})$$ are two expressions. Let $\tilde{P}_{\omega}$ and $\tilde{Q}_{\omega}$ denote the components of $P_{\omega}, Q_{\omega}$ which are linear combinations of $\partial^i U_{n-i,0}$ for $i=0,2,\dots n$. Then $\tilde{P}_{\omega} - \tilde{Q}_{\omega}$ lies in $\cJ$, and hence must be a total derivative. 

Recall next that for $i = 0,1,\dots, \frac{n}{2}$, $$u_{n-i,i} = (\partial^{n-i} W)(\partial^i W) \in \text{gr}(\cV)^{\mathbb{Z}_2} \cong \mathbb{C}[L,\partial L, \partial^2 L, \dots, W, \partial W, \partial^2 W,\dots]^{\mathbb{Z}_2}.$$ We claim that $$u_{n-i,i} = (-1)^i u_{n,0} + \nu,$$ where $\nu$ is a linear combination of $\partial^{j} u_{n-j,0}$ for $j = 2,4,\dots, n$, and hence is a total derivative. This is clear for $i=0$ (taking $\nu = 0$), and since $\partial( u_{n-i,i-1}) = u_{n+1-i, i-1} + u_{n-i,i}$, which is a total derivative, it holds by induction on $i$. It follows from \eqref{w3third} that for $i = 0,1,\dots, \frac{n}{2}$, $$U_{n-i,i} = (-1)^i U_{n,0} + \omega,$$ where $\omega$ is a linear combination of $\partial^{j} U_{n-j,0}$ for $j = 2,4,\dots, n$, and terms in the Virasoro algebra generated by $L$. This proves the claim. \end{proof}

\begin{cor} \label{uniquecoeff} The coefficient of $U_{8,0}$ in \eqref{rel:wt14a} coincides with $$C_8(:U_{0,0} U_{1,1}: - :U_{1,0} U_{1,0}:),$$ and is independent of all choices of normal ordering in $P_8$. Similarly, the coefficient of $U_{10,0}$ in \eqref{rel:wt1622a}-\eqref{rel:wt1631b}, the coefficient of $U_{12,0}$ in \eqref{rel:wt1833a}-\eqref{rel:wt1842b}, and the coefficient of $U_{14,0}$ in \eqref{rel:wt2044a}-\eqref{rel:wt2062b}, are independent of all choices of normally ordering in these expressions.
\end{cor}

Since $U_{0,0} \circ_1$ raises weight by $4$, we may write
$$ U_{0,0} \circ_1 U_{n,0} = F(n,c) U_{n+4,0} + R_n(L,U_{0,0}, U_{2,0},\dots, U_{n+2,0})$$ where $F(n,c)$ denotes the coefficient of $U_{n+4,0}$ and $R_n$ is a normally ordered polynomial in $L,U_{0,0}, U_{2,0},\dots, U_{n+2,0}$ and their derivatives. It is clear from \eqref{w3second} and \eqref{w3third} that $U_{0,0} \circ_1 U_{n,0}$ lies in $\cW^{\mathbb{Z}_2}_{(2)}$, so by Lemma \ref{indep}, we have
\begin{equation} F(n,c) = C_{n+4}(U_{0,0} \circ_1 U_{n,0}).\end{equation}
Similarly, $U_{2,0} \circ_1$ raises weight by $6$, so we may write
$$ U_{2,0} \circ_1 U_{n,0} = G(n,c) U_{n+6,0} + S_n(L,U_{0,0}, U_{2,0},\dots, U_{n+4,0}) ,$$ where $G(n,c)$ denotes the coefficient of $U_{n+6,0}$ and $S_n$ is a normally ordered polynomial in $L,U_{0,0}, U_{2,0},\dots, U_{n+4,0}$ and their derivatives. Then $U_{2,0} \circ_1 U_{n,0}$ lies in $\cW^{\mathbb{Z}_2}_{(2)}$, and
\begin{equation} G(n,c) = C_{n+6}(U_{2,0} \circ_1 U_{n,0}).\end{equation}

The main technical result in this paper is finding the explicit formulas for $F(n,c)$ and $G(n,c)$. A priori, it is not obvious that they should be given by rational functions of $n$ and $c$, but this turns out to be the case.

\begin{thm} \label{thm:explicit1} For all even integers $n\geq 0$, 
$$F(n,c) = -\frac{(10 + n) (p_0(c) + p_1(c) n + p_2(c) n^2 + p_3(c) n^3)}{36 (22 + 5 c) (1 + n) (3 + n) (4 + n)},$$ where 
$$p_0(c) = 720 + 384 c + 12 c^2,\qquad p_1(c) =  -5286 + 125 c + 19 c^2,$$ $$ p_2(c)= -2160 + 40 c + 8 c^2,\qquad p_3(c) = -186 + 11 c + c^2$$
\end{thm}

\begin{thm} \label{thm:explicit2} For all even integers $n\geq 0$, 
$$G(n,c) = -\frac{(12 + n) (q_0(c) + q_1(c) n + q_2(c) n^2 + q_3(c) n^3 + q_4(c) n^4)}{1260 (22 + 5 c) (1 + n) (3 + n) (4 + n) (5 + n)},$$ where
$$q_0(c) = -466200 + 20580 c + 2100 c^2,\qquad q_1(c) = -183780 - 46096 c + 3745 c^2 ,$$ $$ q_2(c) = -74076 - 31732 c + 2065 c^2,\qquad 
 q_3(c) = -19116 - 5624 c + 455 c^2,$$ $$q_4(c) = -1308 - 248 c + 35 c^2.$$
\end{thm}

Using these formulas, we will prove the following result in Section \ref{sect:commonzero}.

\begin{thm} \label{thm:commonzero} For all $c \neq -\frac{22}{5}$ and all even integers $n \geq 16$, we have either $F(n-4,c)\neq 0$ or $G(n-6,c) \neq 0$. In other words, the variety $V \subset \mathbb{C}^2$ determined by $F(n-4,c) = 0$ and $G(n-6,c) = 0$, has no points $(c,n)$ with $n\geq 16$ an even integer.
\end{thm}

Assuming these results for the moment, we have the following
\begin{cor} \label{cor:existence}
\begin{enumerate}
\item For all $c \neq -\frac{22}{5}, \frac{559 \pm 7 \sqrt{76657}}{95}$ and all even integers $n \geq 8$, there exists a decoupling relation 
\begin{equation} \label{mainthm:decoup1} U_{n,0} = P_n(L,U_{0,0},U_{2,0}, U_{4,0}, U_{6,0}),\end{equation} where $P_n$ is a normally ordered polynomial in $L,U_{0,0},U_{2,0},U_{4,0}, U_{6,0}$ and their derivatives.
\item For $c = \frac{559 \pm 7 \sqrt{76657}}{95}$ and all even integers $n \geq 10$, there exists a decoupling relation 
\begin{equation} \label{mainthm:decoup2} U_{n,0} = Q_n(L,U_{0,0},U_{2,0}, U_{4,0}, U_{6,0} ,U_{8,0}),\end{equation} where $Q_n$ is a normally ordered polynomial in $L,U_{0,0},U_{2,0}, U_{4,0}, U_{6,0} ,U_{8,0}$ and their derivatives.
\end{enumerate}
\end{cor}

\begin{proof} Suppose first that $c \neq -\frac{22}{5}, \frac{559 \pm 7 \sqrt{76657}}{95}$. We have the desired relations \eqref{mainthm:decoup1} for $n = 8, 10, 12,14$, so let $n\geq 16$ and assume the result for all even integers $ 8 \leq m<n$. Suppose first that $F(n-4,c) \neq 0$. Applying $U_{0,0}\circ_1$ to both sides of $$U_{n-4,0} =P_{n-4}(L,U_{0,0},U_{2,0},U_{4,0}, U_{6,0})$$ yields
$$F(n-4,c) U_{n,0} + R_{n-4}(L,U_{0,0}, U_{2,0},\dots, U_{n-2,0}) = U_{0,0} \circ_1 P_{n-4}(L,U_{0,0},U_{2,0},U_{4,0}, U_{6,0}).$$ Clearly $U_{0,0} \circ_1 P_{n-4}$ is a normally ordered polynomial in $L, U_{0,0},U_{2,0},\dots, U_{10,0}$ and their derivatives. Using the previous decoupling relations, we can eliminate all occurrences of $U_{8,0}, U_{10,0},\dots, U_{n-2,0}$ and their derivatives, so we get the desired relation. 

If $F(n-4,c) = 0$, then $G(n-6,c) \neq 0$ by assumption. Apply $U_{2,0} \circ_1$ to both sides of $$U_{n-6,0} = P_{n-6}(L,U_{0,0},U_{2,0}, U_{4,0}, U_{6,0}),$$ obtaining
$$G(n-6,c) U_{n,0} + S_{n-6}(L,U_{0,0}, U_{2,0},\dots, U_{n-2,0}) = U_{2,0} \circ_1 P_{n-6}(L,U_{0,0},U_{2,0}, U_{4,0}, U_{6,0}).$$ The right hand side depends only on $L,U_{0,0}, U_{2,0},\dots, U_{12,0}$, so we can use the previous relations to eliminate all occurrences of $U_{8,0}, U_{10,0},\dots, U_{n-2,0}$ and their derivatives.

Finally, suppose that $ c = \frac{559 \pm 7 \sqrt{76657}}{95}$. We have the desired relations \eqref{mainthm:decoup2} for $n = 10, 12, 14$, so let $n\geq 16$ and assume the result for all even integers $10 \leq m<n$. The rest of the proof is the same as above.
\end{proof}

Since $\cW^{\mathbb{Z}_2}$ is strongly generated by $\{L, U_{2n,0}|\ n\geq 0\}$, this immediately implies
\begin{thm} \begin{enumerate} 
\item For all $c \neq -\frac{22}{5}, \frac{559 \pm 7 \sqrt{76657}}{95}$, $\cW^{\mathbb{Z}_2}$ is of type $\cW(2,6,8,10,12)$ with minimal strong generating set $\{L, U_{0,0}, U_{2,0}, U_{4,0}, U_{6,0}\}$.
\item For $c =  \frac{559 \pm 7 \sqrt{76657}}{95}$, $\cW^{\mathbb{Z}_2}$ is of type $\cW(2,6,8,10,12,14)$ with minimal strong generating set $\{L, U_{0,0}, U_{2,0}, U_{4,0}, U_{6,0}, U_{8,0}\}$.
\end{enumerate}
\end{thm}

We will give the proof of Theorem \ref{thm:explicit1} in Section \ref{sect:explicit1}, but we omit the proof of Theorem \ref{thm:explicit2} since it is similar. We prove Theorem \ref{thm:commonzero} in Section \ref{sect:commonzero}. Finally, we will show in Section \ref{sect:22/5} that for $c = -\frac{22}{5}$, $\cW^{\mathbb{Z}_2}$ is also of type $\cW(2,6,8,10,12)$ with minimal strong generating set $\{L, U_{0,0}, U_{2,0}, U_{4,0}, U_{6,0}\}$. This completes the proof of Theorem \ref{intro:mainthm}.

\section{Proof of Theorem \ref{thm:explicit1}} \label{sect:explicit1}
For all $n\geq 0$, we have
$$U_{0,0} \circ_1 U_{n,0} = \ : (U_{0,0} \circ_1 \partial^n W) W: + (U_{0,0} \circ_0 \partial^n W) \circ_0 W + :(\partial^n W) (U_{0,0} \circ_1 W):,$$ so in order to compute
$F(n,c) = C_{n+4}(U_{0,0} \circ_1 U_{n,0})$ we need to calculate the following three expressions:
\begin{equation} \label{piece1} C_{n+4}\bigg(: (U_{0,0} \circ_1 \partial^n W) W:\bigg),\end{equation}
\begin{equation} \label{piece2} C_{n+4} \bigg((U_{0,0} \circ_0 \partial^n W) \circ_0 W\bigg),\end{equation}
\begin{equation} \label{piece3} C_{n+4} \bigg( :(\partial^n W) (U_{0,0} \circ_1 W):\bigg).\end{equation}

\begin{lemma} For all $n\geq 1$, $C_{n+4} \bigg((U_{0,0} \circ_0 \partial^n W) \circ_0 W\bigg) = 0$.
\end{lemma}
\begin{proof} We have $U_{0,0} \circ_0 \partial^n W = \partial \big(U_{0,0} \circ_0 \partial^{n-1} W\big)$, so 
$$ (U_{0,0} \circ_0 \partial^n W) \circ_0 W = \partial \big(U_{0,0} \circ_0 \partial^{n-1} W\big) \circ_0 W = 0.$$ \end{proof}

To compute \eqref{piece1}, we begin with the following observation.
\begin{lemma} For all $n\geq 1$,
$$U_{0,0} \circ_0 \partial^{n-1} W - \frac{64}{22 + 5 c}  \partial^n \big(:L L W:\big)  + \frac{64}{22 + 5 c} \partial^{n-1} \big(:(\partial L) L W:\big)$$
$$- \frac{10 (14 + c)}{3 (22 + 5 c)}  \partial^{n+2} \big(:LW:\big) + \frac{86 + 5 c}{22 + 5 c}  \partial^{n+1} \big( :(\partial L) W:\big)$$
$$- \frac{26 + 3 c}{22 + 5 c} \partial^{n} \big( :(\partial^2 L) W:\big) + \frac{2 (-2 + c)}{3 (22 + 5 c)} \partial^{n-1} \big( :(\partial^3 L) W:\big)$$
$$- \frac{-186 + 11 c + c^2}{36 (22 + 5 c)}  \partial^{n+4} W = 0.$$
\end{lemma}

\begin{proof} This is easy to verify for $n=1$ and follows immediately from the fact that $$\partial \big( U_{0,0} \circ_0 \partial^{n-1} W\big) = U_{0,0} \circ_0 \partial^n W.$$
\end{proof}

\begin{lemma} \label{cor:firstformula} For all $n\geq 1$,
$$U_{0,0} \circ_1 \partial^n W  -\frac{64 (1 + n)}{22 + 5 c} \partial^n \big(:LLW:\big)  +\frac{64 n}{22 + 5 c}  \partial^{n-1} \big(:(\partial L) L W:\big)$$ 
$$ -\frac{2 (258 + 15 c + 70 n + 5 c n)}{3 (22 + 5 c)}   \partial^{n+2} \big(:LW:\big) + \frac{236 + 10 c + 86 n + 5 c n}{22 + 5 c}  \partial^{n+1} \big( :(\partial L) W:\big) $$
$$-\frac{58 + 3 c + 26 n + 3 c n}{22 + 5 c} \partial^{n} \big( :(\partial^2 L) W:\big) + \frac{2 (-2 + c)}{3 (22 + 5 c)} \partial^{n-1} \big( :(\partial^3 L) W:\big) $$
$$ -\frac{-426 + 91 c + 5 c^2 - 186 n + 11 c n + c^2 n}{36 (22 + 5 c)} \partial^{n+4} W = 0.$$ \end{lemma}

\begin{proof} For $n=1$ this can be checked directly. It follows by induction on $n$ using the previous lemma and the formula $$ \partial \big(U_{0,0} \circ_1 \partial^{n-1} W \big) = - U_{0,0} \circ_0 \partial^{n-1} W + U_{0,0} \circ_1 \partial^{n} W.$$
\end{proof}

\begin{cor} \label{cor:firstformula1} For all $n\geq 1$,
$$C_{n+4}\bigg((:U_{0,0} \circ_1 \partial^n W)W:\bigg)  -\frac{64 (1 + n)}{22 + 5 c} C_{n+4} \bigg( :(\partial^n (:LLW:))W:\bigg)$$
$$+\frac{64 n}{22 + 5 c}  C_{n+4}\bigg(:\big(\partial^{n-1} \big(:(\partial L) L W:\big)\big)W:\bigg) $$ 
$$ -\frac{2 (258 + 15 c + 70 n + 5 c n)}{3 (22 + 5 c)}  C_{n+4}\bigg(:(\partial^{n+2} (:LW:))W:\bigg)  $$
$$+ \frac{236 + 10 c + 86 n + 5 c n}{22 + 5 c} C_{n+4}\bigg(:\partial^{n+1} \big( :(\partial L) W:\big) W:\bigg) $$
$$-\frac{58 + 3 c + 26 n + 3 c n}{22 + 5 c} C_{n+4}\bigg(: \partial^{n} \big( :(\partial^2 L) W:\big) W:\bigg)$$
$$+ \frac{2 (-2 + c)}{3 (22 + 5 c)}C_{n+4}\bigg(: \partial^{n-1} \big( :(\partial^3 L) W:\big) W:\bigg)$$
$$ -\frac{-426 + 91 c + 5 c^2 - 186 n + 11 c n + c^2 n}{36 (22 + 5 c)} = 0.$$ \end{cor}

Next, we need the following calculations:
$$C_{n+4, 0} \bigg( :(\partial^n (:LLW:))W:\bigg)= \frac{15}{(1 + n) (2 + n) (3 + n) (4 + n)},$$
$$C_{n+4, 1} \bigg(:(\partial^n (:LLW:))W:\bigg)  = \frac{7}{(1 + n) (2 + n) (3 + n)},$$
$$C_{n+4, 2} \bigg(:(\partial^n (:LLW:))W: \bigg) = \frac{1}{(1 + n) (2 + n)},$$ 
$$C_{n+4, i} \bigg(:(\partial^n (:LLW:))W: \bigg) = 0,\qquad 3 \leq i \leq \frac{n+4}{2}.$$ By \eqref{defcn}, we have \begin{equation} \label{first:1} C_{n+4} \bigg( :(\partial^n (:LLW:))W:\bigg) = \frac{-1 + n}{(2 + n) (3 + n) (4 + n)}.\end{equation}

Next, we have 
$$C_{n+4, 0}\bigg(:\big(\partial^{n-1} \big(:(\partial L) L W:\big)\big)W:\bigg) = -\frac{24}{n (1 + n) (2 + n) (3 + n) (4 + n)},$$
$$C_{n+4, 1}\bigg(:\big(\partial^{n-1} \big(:(\partial L) L W:\big)\big)W:\bigg) = -\frac{10}{n (1 + n) (2 + n) (3 + n)},$$
$$C_{n+4, 2}\bigg(:\big(\partial^{n-1} \big(:(\partial L) L W:\big)\big)W:\bigg)= -\frac{1}{n (1 + n) (2 + n)}, $$ 
$$C_{n+4, i}\bigg(:\big(\partial^{n-1} \big(:(\partial L) L W:\big)\big)W:\bigg) = 0,\qquad 3\leq i \leq  \frac{n+4}{2}.$$ Therefore 
\begin{equation} \label{first:2} C_{n+4}\bigg(:\big(\partial^{n-1} \big(:(\partial L) L W:\big)\big)W:\bigg) = -\frac{-4 + n}{n (2 + n) (3 + n) (4 + n)}.\end{equation}

Next, we have 
$$C_{n+4, 0}\bigg(:(\partial^{n+2} (:LW:))W:\bigg)  =\frac{3}{(3 + n) (4 + n)},$$
$$C_{n+4, 1}\bigg(:(\partial^{n+2} (:LW:))W:\bigg) = \frac{1}{3 + n},$$
$$C_{n+4, i}\bigg(:(\partial^{n+2} (:LW:))W:\bigg) = 0,\qquad 2 \leq i \leq  \frac{n+4}{2}.$$ Therefore
\begin{equation} \label{first:3} C_{n+4}\bigg(:(\partial^{n+2} (:LW:))W:\bigg) = -\frac{1 + n}{(3 + n) (4 + n)}.\end{equation}

Next, we have 
$$C_{n+4,0}\bigg(:\partial^{n+1} \big( :(\partial L) W:\big) W:\bigg) = -\frac{6}{(2 + n) (3 + n) (4 + n)},$$
$$C_{n+4,1}\bigg(:\partial^{n+1} \big( :(\partial L) W:\big) W:\bigg) = -\frac{1}{(2 + n) (3 + n)},$$ 
$$C_{n+4,i}\bigg(:\partial^{n+1} \big( :(\partial L) W:\big) W:\bigg) = 0, \qquad 2\leq i \leq \frac{n+4}{2}.$$ Therefore
\begin{equation} \label{first:4} C_{n+4}\bigg(:\partial^{n+1} \big( :(\partial L) W:\big) W:\bigg) = \frac{-2 + n}{(2 + n) (3 + n) (4 + n)}.\end{equation}

Next, we have
$$C_{n+4, 0}\bigg(: \partial^{n} \big( :(\partial^2 L) W:\big) W:\bigg)= \frac{18}{(1 + n) (2 + n) (3 + n) (4 + n)},$$
$$C_{n+4, 1}\bigg(: \partial^{n} \big( :(\partial^2 L) W:\big) W:\bigg)=  \frac{2}{(1 + n) (2 + n) (3 + n)},$$
$$C_{n+4, i}\bigg(: \partial^{n} \big( :(\partial^2 L) W:\big) W:\bigg)= 0, \qquad 2\leq i \leq  \frac{n+4}{2}.$$ Therefore
\begin{equation} \label{first:5} C_{n+4}\bigg(: \partial^{n} \big( :(\partial^2 L) W:\big) W:\bigg) = -\frac{2 (-5 + n)}{(1 + n) (2 + n) (3 + n) (4 + n)}.\end{equation}

Next, we have
$$C_{n+4, 0}\bigg(: \partial^{n-1} \big( :(\partial^3 L) W:\big) W:\bigg) = -\frac{72}{n (1 + n) (2 + n) (3 + n) (4 + n)},$$
$$C_{n+4, 1}\bigg(: \partial^{n-1} \big( :(\partial^3 L) W:\big) W:\bigg)  = -\frac{6}{n (1 + n) (2 + n) (3 + n)},$$
$$C_{n+4, i}\bigg(: \partial^{n-1} \big( :(\partial^3 L) W:\big) W:\bigg) = 0,\qquad 2\leq i \leq  \frac{n+4}{2}.$$ Therefore
\begin{equation} \label{first:6} C_{n+4}\bigg(: \partial^{n-1} \big( :(\partial^3 L) W:\big) W:\bigg)  = \frac{6 (-8 + n)}{n (1 + n) (2 + n) (3 + n) (4 + n)}.\end{equation}
The explicit formula for \eqref{piece1} is obtained by combining \eqref{first:1}-\eqref{first:6} with Corollary \ref{cor:firstformula1}.

To find the explicit formula for \eqref{piece3}, we need the following calculation.
\begin{equation} \label{secondformula} \begin{split} U_{0,0} \circ_1 W  -  \frac{64}{22 + 5 c} :LLW: - \frac{2 (258 + 15 c)}{3 (22 + 5 c)} \partial^2 (:L W:) + \frac{236 + 10 c}{22 + 5 c} \partial \big(:(\partial L)W:\big)  
\\ - \frac{58 + 3 c}{22 + 5 c} :(\partial^2 L)W: - \frac{-426 + 91 c + 5 c^2}{36 (22 + 5 c)} \partial^4 W =0.\end{split} \end{equation} Since $C_{n+4}\big(:(\partial^n W)(\partial^4 W):\big) = 1$ when $n$ is even, this immediately implies

\begin{cor}  \label{cor:secondformula} We have $$C_{n+4}\bigg( :(\partial^n W)(U_{0,0} \circ_1 W): \bigg)  -  \frac{64}{22 + 5 c}  C_{n+4}\bigg(:(\partial^n W) (:LLW:):\bigg) $$ $$- \frac{2 (258 + 15 c)}{3 (22 + 5 c)} C_{n+4}\bigg(:(\partial^n W) \big(\partial^2 (:L W:)\big):\bigg)$$
 $$ + \frac{236 + 10 c}{22 + 5 c} C_{n+4}\bigg(:(\partial^n W) \big(\partial (:(\partial L) W:)\big):\bigg) $$
$$ - \frac{58 + 3 c}{22 + 5 c} C_{n+4}\bigg(:(\partial^n W) \big(:(\partial^2 L) W:\big):\bigg) $$
 $$ - \frac{-426 + 91 c + 5 c^2}{36 (22 + 5 c)} = 0.$$
\end{cor}

We calculate
$$C_{n+4,0}\bigg(:(\partial^n W) (:LLW:):\bigg) = \frac{-1 + n}{(2 + n) (3 + n) (4 + n)},$$ 
$$C_{n+4,i}\bigg(:(\partial^n W) (:LLW:):\bigg) = 0,\qquad 1\leq i \leq \frac{n+4}{2}.$$ Therefore 
\begin{equation} \label{third:1} C_{n+4}\bigg(:(\partial^n W) (:LLW:):\bigg) = \frac{-1 + n}{(2 + n) (3 + n) (4 + n)}.\end{equation}

Next, we have
$$C_{n+4,0}\bigg(:(\partial^n W) \big(\partial^2 (:L W:)\big):\bigg) = -\frac{2 (-5 + n)}{(1 + n) (2 + n) (3 + n) (4 + n)},$$ 
$$C_{n+4,1}\bigg(:(\partial^n W) \big(\partial^2 (:L W:)\big):\bigg) = -\frac{2 (-3 + n)}{(1 + n) (2 + n) (3 + n)},$$ 
$$C_{n+4,2}\bigg(:(\partial^n W) \big(\partial^2 (:L W:)\big):\bigg) = -\frac{-1 + n}{(1 + n) (2 + n)},$$
$$C_{n+4,i}\bigg(:(\partial^n W) \big(\partial^2 (:L W:)\big):\bigg) =0,\qquad 3 \leq i \leq  \frac{n+4}{2}.$$ Therefore
\begin{equation} \label{third:2} C_{n+4}\bigg(:(\partial^n W) \big(\partial^2 (:L W:)\big):\bigg) = -\frac{1 + n}{(3 + n) (4 + n)} .\end{equation}

Next, we have
$$C_{n+4,0}\bigg(:(\partial^n W) \big(\partial (:(\partial L) W:)\big):\bigg)= -\frac{2 (-5 + n)}{(1 + n) (2 + n) (3 + n) (4 + n)},$$
$$C_{n+4,1}\bigg(:(\partial^n W) \big(\partial (:(\partial L) W:)\big):\bigg) = -\frac{-3 + n}{(1 + n) (2 + n) (3 + n)}, $$
$$C_{n+4,i}\bigg(:(\partial^n W) \big(\partial (:(\partial L) W:)\big):\bigg) = 0, \qquad 2 \leq i \leq  \frac{n+4}{2}.$$ Therefore
\begin{equation} \label{third:3}C_{n+4}\bigg(:(\partial^n W) \big(\partial (:(\partial L) W:)\big):\bigg) = \frac{-2 + n}{(2 + n) (3 + n) (4 + n)}. \end{equation}

Next, we have
$$C_{n+4,0}\bigg(:(\partial^n W) \big(:(\partial^2 L) W:\big):\bigg) = -\frac{2 (-5 + n)}{(1 + n) (2 + n) (3 + n) (4 + n)},$$
$$C_{n+4,i}\bigg(:(\partial^n W) \big(:(\partial^2 L) W:\big):\bigg) = 0, \qquad 1\leq i \leq \frac{n+4}{2}.$$ Therefore 
\begin{equation} \label{third:4} C_{n+4}\bigg(:(\partial^n W) \big(:(\partial^2 L) W:\big):\bigg) = -\frac{2 (-5 + n)}{(1 + n) (2 + n) (3 + n) (4 + n)}.\end{equation}

The explicit formula for \eqref{piece3} is obtained by combining \eqref{third:1}-\eqref{third:4} with Corollary \ref{cor:secondformula}. Finally, combining the formulas for \eqref{piece1} and \eqref{piece3} completes the proof of Theorem \ref{thm:explicit1}. The proof of Theorem \ref{thm:explicit2} is similar and is omitted.

\section{Proof of Theorem \ref{thm:commonzero}} \label{sect:commonzero}
First, set $$f(n,c) = p_0(c) + p_1(c) n + p_2(c) n^2 + p_3(c) n^3,$$ $$g(n,c) = q_0(c) + q_1(c) n + q_2(c) n^2 + q_3(c) n^3 + q_4(c) n^4,$$ where $p_i(c)$ and $q_i(z)$ are as in Theorems \ref{thm:explicit1} and \ref{thm:explicit2}. Clearly when $n$ is a positive integer, $$F(n,c) = 0 \Leftrightarrow f(n,c)=0,\qquad G(n,c) = 0 \Leftrightarrow g(n,c)=0.$$
We may regard $f(n,c)$ as a family of quadratics in $c$ parametrized by $n$, namely,
$$f(n,c) =(720 - 5286 n  -  2160 n^2 - 186 n^3) + (384 + 125 n + 40 n^2 + 11 n^3) c + (12   + 19 n   + 8  n^2   +  n^3)c^2.$$ Using the quadratic formula, we can express the roots $r_1(n)$ and $r_2(n)$ as functions of $n$. Since $$\lim_{n\ra \infty} \frac{1}{n^3} f(n,c) = p_3(c),$$ we have $$\lim_{n\ra \infty} r_1(n) =  \frac{-11 - \sqrt{865}}{2} \sim -20.2054,\qquad \lim_{n\ra \infty} r_2(n) =  \frac{-11 +  \sqrt{865}}{2} \sim 9.20544.$$
Similarly, we regard $g(n,c)$ as a family of quadratics in $c$ parametrized by $n$, namely
$$g(n,c) = (-466200 - 183780 n - 74076 n^2 - 19116 n^3- 1308 n^4) $$ $$ + (20580 - 46096 n - 31732 n^2- 5624 n^3 - 248 n^4)c $$
$$+ (2100  +  3745  n  + 2065  n^2   + 455 n^3   + 35 n^4) c^2,$$
and we can express the roots $s_1(n)$ and $s_2(n)$ as functions of $n$. Since $$\lim_{n\ra \infty} \frac{1}{n^4} g(n,c) = q_4(c),$$ we have $$\lim_{n\ra \infty} s_1(n) = \frac{2 (62 - \sqrt{15289})}{35} \sim -3.52278,\qquad \lim_{n\ra \infty} s_2(n) = \frac{2 (62 +  \sqrt{15289})}{35}  \sim 10.6085.$$

For $i=1,2$ and $n$ regarded as a positive real variable, $r_i(n)$ and $s_i(n)$ are differentiable functions of $n$. By computing the derivatives of $r_1(n)$ and $r_2(n)$, we see that both are decreasing functions on $(9,\infty)$. We have 
$$r_1(22) =  \frac{-139622 - 2 \sqrt{51839598721}}{29900}  \cong -19.8993,$$ so $ -20.2054< r_1(n) <  -19.8993$ for all $n >22$. 
Similarly,
$$r_2(22) = \frac{-139622 + 2 \sqrt{51839598721}}{29900}  \cong 10.56,$$
so $9.20544 < r_2(n)<10.56$ for all $n>22$. This implies that if $n > 26$ is an even positive integer and $F(n-4,c) = 0$, $c$ is a real number that lies either in $(-20.2054,  -19.8993)$ or $(9.20544 , 10.56)$.

Similarly, both $s_1(n)$ and $s_2(n)$ are decreasing functions on $(7,\infty)$. Note that $$s_1(20)  = \frac{24566535 - 945 \sqrt{1800197569}}{5071500}  \cong  -3.06194,$$ so $-3.52278< s_1(n) <  -3.06194$ for all $n >20$. 
Likewise, $$s_2(20) =  \frac{24566535 + 945 \sqrt{1800197569}}{5071500}  \cong 12.75,$$ so $10.6085 < s_2(n)<12.75$ for all $n>20$.
Therefore if $n > 26$ is a positive integer and $G(n-6,c) = 0$, then $c$ is a real number lying either in $(-3.52278, -3.06194)$ or $(10.6085 ,12.75)$. This shows that Theorem \ref{thm:commonzero} holds for all $n > 26$. It is straightforward to verify it for $16 \leq n \leq 26$, which completes the proof.

\section{The case $c = -\frac{22}{5}$} \label{sect:22/5}

In this case, the rescaled generator $W$ satisfies \eqref{ope:c=-22/5}, and the generators for the orbifold $\cW^{\mathbb{Z}_2}$ are still $\{L, U_{2n,0}|\ n\geq 0\}$. One can check that the relations $$U_{8,0} = P_8(L,U_{0,0} ,U_{2,0},U_{4,0}, U_{6,0}),\qquad U_{10,0} = P_{10}(L,U_{0,0} ,U_{2,0},U_{4,0}, U_{6,0})$$ both exist. Also, we have
$$U_{0,0} \circ_1 U_{n,0} = F(n) U_{n+4,0} + P,$$ where $P$ is a normally ordered polynomial in $L, U_{0,0}, U_{2,0},\dots, U_{n+2,0}$ and their derivatives.
Using similar methods to the proof of Theorem \ref{thm:explicit1}, one can show that
\begin{equation} F(n) = -\frac{64 (6 + n) (10 + n) (1 + 7 n)}{75 (1 + n) (3 + n)},$$ which is exactly $$\lim_{c \ra -\frac{22}{5}} (22+5c) F(n,c).\end{equation}
Starting from the decoupling relations for $U_{8,0}$ and $U_{10,0}$ and applying $U_{0,0} \circ_1$ repeatedly, by the same argument as the proof of Corollary \ref{cor:existence}, we can construct decoupling relations

$$U_{n,0} = P_n(L,U_{0,0} ,U_{2,0},U_{4,0}, U_{6,0}),\qquad n = 12, 14, \dots.$$ 
We obtain
\begin{thm} For $c = -\frac{22}{5}$, $\cW^{\mathbb{Z}_2}$ is of type $\cW(2,6,8,10,12)$ with minimal strong generating set $\{L, U_{0,0}, U_{2,0}, U_{4,0}, U_{6,0}\}$.
\end{thm}

\section{Appendix} \label{sect:appendix}
In this Appendix, we write down the explicit normally ordered polynomial relation in weight $14$, which is unique up to scalar multiplies.
$$:U_{0,0} U_{1,1}:  - :U_{1,0} U_{1,0}: + \frac{40}{22 + 5 c} :L L U_{4,0}: + \frac{144}{22 + 5 c}:(\partial^2 L) L U_{2,0}: $$
$$+ \frac{144}{22 + 5 c} :(\partial L) (\partial L) U_{2,0}: + \frac{112}{22 + 5 c} :(\partial L) L \partial U_{2,0}:  -\frac{56}{22 + 5 c} :L L \partial^2 U_{2,0}: $$
$$ -\frac{8}{22 + 5 c} :(\partial^3 L) L \partial U_{0,0}: -\frac{24}{22 + 5 c} :(\partial^2 L)(\partial L) \partial U_{0,0}: -\frac{64}{22 + 5 c}  :(\partial^2 L) L \partial^2 U_{0,0}: $$
$$-\frac{64}{22 + 5 c} :(\partial L) (\partial L)\partial^2 U_{0,0}: -\frac{48}{22 + 5 c} :(\partial L) L \partial^3 U_{0,0}:  + \frac{28}{3 (22 + 5 c)} :L L \partial^4 U_{0,0}: $$
$$+ \frac{47}{180} :L U_{6,0}: + \frac{5 (54 + 13 c)}{12 (22 + 5 c)} :(\partial^2 L) U_{4,0}: + \frac{1962 + 155 c}{24 (22 + 5 c)} :(\partial L)\partial U_{4,0}:$$
$$-\frac{-182 + 75 c}{24 (22 + 5 c)} :L \partial^2 U_{4,0}: + \frac{3 (-2 + c)}{2 (22 + 5 c)} :(\partial^4 L) U_{2_0}: -\frac{170 - 53 c}{24 (22 + 5 c)} :(\partial^3 L) \partial U_{2,0}:$$
$$-\frac{486 + 61 c}{8 (22 + 5 c)} :(\partial^2 L) \partial^2 U_{2,0}: -\frac{1878 + 125 c}{12 (22 + 5 c)} :(\partial L) \partial^3 U_{2,0}: -\frac{662 - 75 c}{24 (22 + 5 c)} :L \partial^4 U_{2,0}:$$
$$-\frac{8}{15 (22 + 5 c)}  :(\partial^6 L) U_{0,0}: -\frac{218 + 3 c}{48 (22 + 5 c)} :(\partial^5 L)\partial U_{0,0}: -\frac{2 (15 + c)}{3 (22 + 5 c)} :(\partial^4 L) \partial^2 U_{0,0}:$$
$$-\frac{1966 + 137 c}{144 (22 + 5 c)} :(\partial^3 L) \partial^3 U_{0,0}:
+ \frac{102 + 61 c}{48 (22 + 5 c)}  :(\partial^2 L) \partial^4 U_{0,0}:
+ \frac{25 (14 + c)}{12 (22 + 5 c)} :(\partial L) \partial^5 U_{0,0}:$$
$$+ \frac{662 - 75 c}{120 (22 + 5 c)}  :L \partial^6 U_{0,0}:
-\frac{896}{15 (22 + 5 c)^2} : (\partial^6 L) L L L:
-\frac{256}{5 (22 + 5 c)^2)}  :(\partial^5 L) (\partial L) L L:$$
$$-\frac{1664}{3 (22 + 5 c)^2}  :(\partial^4 L)(\partial^2  L) L L:
-\frac{5504}{9 (22 + 5 c)^2} :(\partial^3 L)(\partial^3 L) L L: $$
$$+\frac{5632}{3 (22 + 5 c)^2}  :(\partial^4 L) (\partial L)(\partial  L) L: 
+ \frac{4352}{(22 + 5 c)^2} :(\partial^3 L)(\partial^2 L)(\partial L) L: $$
$$+ \frac{1024}{(22 + 5 c)^2} :(\partial^2 L) (\partial^2 L) (\partial^2 L) L: 
+\frac{4096}{3 (22 + 5 c)^2}  :(\partial^3 L) (\partial L)  (\partial L) \partial L: $$
$$+\frac{896}{(22 + 5 c)^2} :(\partial^2 L) (\partial^2 L)  (\partial L)\partial L:
-\frac{29486 - 2263 c}{630 (22 + 5 c)^2} :(\partial^8 L) L L:$$
$$+\frac{32 (-5174 + 209 c)}{315 (22 + 5 c)^2} :(\partial^7 L) (\partial L) L:
+ \frac{2 (-28198 + 2427 c)}{45 (22 + 5 c)^2} :(\partial^6 L)(\partial^2 L) L:$$
$$+ \frac{32 (-1174 + 109 c)}{15 (22 + 5 c)^2} :(\partial^5 L)(\partial^3 L) L:
-\frac{14486 - 1307 c}{9 (22 + 5 c)^2}  :(\partial^4 L)(\partial^4 L)L:$$
$$-\frac{2 (25518 + 2065 c)}{45 (22 + 5 c)^2}  :(\partial^6 L)(\partial L)\partial L: 
-\frac{32 (541 + 52 c)}{5 (22 + 5 c)^2}  :(\partial^5 L)(\partial^2 L)\partial L:$$
$$-\frac{104 (482 + 37 c)}{9 (22 + 5 c)^2} :(\partial^4 L)(\partial^3 L)\partial L: 
-\frac{2 (-286 + 167 c)}{3 (22 + 5 c)^2)} :(\partial^4 L)(\partial^2 L) \partial^2 L:$$
$$-\frac{8 (-886 + 23 c)}{9 (22 + 5 c)^2}  :(\partial^3 L) (\partial^3 L)\partial^2 L:$$
$$-\frac{-342897348 - 25407820 c + 402775 c^2}{5443200 (22 + 5 c)^2} :(\partial^{10} L) L: $$
$$-\frac{-345995076 - 26686756 c + 626275 c^2}{544320 (22 + 5 c)^2}  :(\partial^9 L)\partial L: $$
$$-\frac{-349360452 - 27205180 c + 577903 c^2}{120960 (22 + 5 c)^2)} :(\partial^8 L)\partial^2 L:$$
$$-\frac{-2804245644 - 218591252 c + 4546349 c^2}{362880 (22 + 5 c)^2} :(\partial^7 L)\partial^3 L: $$
$$-\frac{-21995034 - 1714285 c + 35605 c^2}{1620 (22 + 5 c)^2} :(\partial^6 L)\partial^4 L:$$
$$-\frac{-140780292 - 10970908 c + 228175 c^2}{17280 (22 + 5 c)^2} :(\partial^5 L) \partial^5 L:$$
$$-\frac{93733420 - 225352108 c - 18450565 c^2 + 381800 c^3}{479001600 (22 + 5 c)^2} \partial^{12} L$$
$$-\frac{181248 + 5590 c - 475 c^2}{60480 (22 + 5 c)} U_{8,0}  -\frac{-63456 - 3862 c + 115 c^2}{4320 (22 + 5 c)} \partial^2 U_{6,0} $$
$$+ \frac{-74208 - 5206 c + 115 c^2}{1728 (22 + 5 c)} \partial^4 U_{4,0}
-\frac{-74208 - 5270 c + 115 c^2}{1440 (22 + 5 c)} \partial^6 U_{2,0}$$
$$-\frac{1264260 + 89924 c - 1955 c^2}{120960 (22 + 5 c)} \partial^8 U_{0,0} = 0.$$

\end{document}